\documentclass[11pt,oneside]{article}
\usepackage{amsmath,amsthm,amsfonts,amssymb,MnSymbol,mathrsfs,latexsym}
\usepackage[english]{babel}

\usepackage[applemac]{inputenc}
\usepackage[all]{xy}
\numberwithin{equation}{section}

  \newcommand{\const}{\rm const}

\theoremstyle{plain}
\newtheorem{theorem}{Theorem}[section]
\theoremstyle{theorem}

\newtheorem{definition}{Definition}[section]
\newtheorem{remark}{Remark}[section]

\renewenvironment{proof}{{\bf{Proof.}}}{\hfill $\Box$ \\}

\DeclareMathOperator*{\esssup}{ess\,sup}

\usepackage{lipsum}% http://ctan.org/pkg/lipsum
\usepackage{fancyhdr}% http://ctan.org/pkg/fancyhdr

\title{\large \textbf{BOCHNER-RIESZ OPERATORS IN GRAND LEBESGUE SPACES}}

\footnotesize\date{}

\author{\normalsize \textbf{Maria Rosaria FORMICA ${}^{1}$},   \normalsize\textbf{Eugeny OSTROVSKY
${}^2$} and \normalsize\textbf{Leonid SIROTA ${}^3$}}

\begin{document}

\maketitle

\begin{center}
{\footnotesize ${}^{1}$ Universit\`{a} degli Studi di Napoli \lq\lq Parthenope\rq\rq, via Generale Parisi 13,\\
Palazzo Pacanowsky, 80132,
Napoli, Italy.} \\

\vspace{2mm}

{\footnotesize e-mail: mara.formica@uniparthenope.it} \\

\vspace{4mm}

{\footnotesize ${}^{2,\, 3}$  Bar-Ilan University, Department of Mathematics and Statistics, \\
52900, Ramat Gan, Israel.} \\

\vspace{2mm}

{\footnotesize e-mail: eugostrovsky@list.ru}\\

\vspace{2mm}

{\footnotesize e-mail: sirota3@bezeqint.net} \\

\end{center}

 \vspace{3mm}

\begin{abstract}
We provide the conditions for the boundedness of the Bochner-Riesz
operator acting between two different Grand Lebesgue Spaces.
Moreover we obtain a lower estimate for the constant appearing in
the Lebesgue-Riesz norm estimation of the Bochner-Riesz operator and
we investigate the convergence of the Bochner-Riesz approximation in
Lebesgue-Riesz spaces.
%We deduce the conditions for the boundedness of Bochner-Riesz operator acting  between two different Grand Lebesgue Spaces,
%and we obtain upper and lower estimates for the norm of these operators.
\end{abstract}

\vspace{4mm}


\vspace{2mm}


\section{Introduction}\label{Intro}

\vspace{5mm}

 The well-known (pseudo-differential) Bochner-Riesz linear operator $B_R^{\alpha}[f]$, acting on measurable functions $ f: \mathbb{R}^n \to  \mathbb{R}$ ($n>1$), is defined as follows

\begin{equation} \label{Boch}
B_R^{\alpha}[f](t) \stackrel{def}{=} \int_{\mathbb{R}^n} e^{-i(t,y)}
\ \left( \ 1 - \frac{|y|^2}{R^2}  \ \right)^{\alpha}_+ \
\tilde{f}(y) \ dy,
\end{equation}
where
\begin{equation} \label{restrictions}
R = {\const} > 0, \ \ \alpha = {\const}, \ \ z_+  = \max(z,0), \ \ z
\in \mathbb{R},
\end{equation}

\noindent $(x,y)$ stands for the scalar (inner) product of two
vectors $ x,y\in\mathbb{R}^n$,  $ |y|^2 := (y,y)$ and $\tilde{f} \ $
denotes, as ordinary, the Fourier transform

\begin{equation} \label{Fourier}
\tilde{f}(y) = \int_{\mathbb{R}^n}e^{i(x,y)} \, f(x) \, dx.
\end{equation}

\vspace{4mm}

 The applications of these operators are described in particular
 in functional analysis (see \cite{Duoandikoetxea},  \cite{Grafakos}) and in statistics of random processes and field (see \cite{Buldygin Ill},
\cite{Ostrovsky6}).

\vspace{3mm}

 The classical Lebesgue-Riesz norm $ ||f||_p, \  p \in [1,\infty]$, for the function $f$,
 is denoted by
$$
||f||_p := \left[ \ \int_{\mathbb{R}^n} |f(x)|^p \ dx \
\right]^{1/p}, \ \ \ 1 \le p < \infty,
$$

$$
||f||_{\infty} := \esssup_{x \in\mathbb{R}^n} |f(x)|, \ \ \ p=\infty
$$
 and the corresponding Banach space is, as usually,

 $$
   L_p = L_p(\mathbb{R}^n) = \{  \ f: \ ||f||_p < \infty \ \}.
 $$

\vspace{4mm}

 Denote, for an arbitrary linear or quasi linear operator $U$, acting from $ L_p$ to $ L_q, \ p,q \ge 1 $, its norm
$$
||U||_{p,q} = ||U||_{p \to q} := \sup_{0 \ne f \in L_p} \left[ \ \frac{||Uf||_q}{||f||_p} \  \right].
$$
 Of course, the operator $ U$ is bounded as the operator acting from the space $ L_p  $ into the space $ L_q $
 iff $ ||U||_{p,q} < \infty $.

 More generally, for the operator $U $  acting from some Banach
space $ F$ equipped with the norm $ ||\cdot||_F $ into another (in
the general case) Banach space $D $,  having the norm $||\cdot||_D
$, denote as usually
$$
||U||_{F \to D} := \sup_{0 \ne f \in F} \left[ \
\frac{||Uf||_D}{||f||_F} \ \right].
$$

 \vspace{3mm}

There exists a huge numbers of works devoted to the $ \ p,q \ $
estimates for Bochner-Riesz operators $ \ B_R^{\alpha}, \ $ as a
rule for the case  $ \ q = p, \ $ see e.g. \cite[chapter
5]{Grafakos},
\cite{Cordoba1979,Duoandikoetxea,Fefferman2,Karasev,LuYanBook2013,Stein1,Stein2},
etc.

 \vspace{2mm}

Our aim, in this paper, is to extend some results contained in the
above mentioned works concerning upper estimates for the norm of the
Bochner-Riesz operator, in the case of different Lebesgue-Riesz
spaces (Section \ref{estimates Lebesgue}) and to extend these
estimates to the so-called Grand Lebesgue Spaces (GLS), see Section
\ref{Section Boundedness GLS}.

We deduce also
%non trivial quantitative lower estimates for the norm
%of these operators
a non trivial quantitative lower estimate for the coefficient in the
Lebesgue-Riesz norm estimation for the Bochner-Riesz operator
(Section \ref{Section lower bound BR}) and we study the convergence
of the Bochner-Riesz approximation in Lebesgue-Riesz spaces (Section
\ref{section convergence}).

\vspace{5mm}

\section{Norm estimates for the Bochner-Riesz operator acting between different Lebesgue-Riesz
spaces}\label{estimates Lebesgue}

\vspace{5mm}

 We clarify slightly the known results about the Lebesgue-Riesz $p,r$ norms for the operators $B_R^{\alpha}$ defined in \eqref{Boch},
 see e.g.
\cite{Duoandikoetxea,Grafakos,Fefferman1,Karasev,Stein2}, ecc.

\vspace{3mm}

 It is important to observe that the Bochner-Riesz operator $ \ B_R^{\alpha} $ may be rewritten as a
 convolution, namely

\begin{equation} \label{conv representation}
 B_R^{\alpha}[f](t) =  f*K_{\lambda}^R(t) = \int_{\mathbb{R}^n} f(t-s) \ K_{\lambda}^R (s) \ ds,
\end{equation}
where $ R = {\const} > 0 $,

$$
K_{\lambda}^R(z) =  C(\alpha) \ R^n \
\frac{J_{\lambda}(R|z|)}{|Rz|^{\lambda}}, \  \ \ \lambda =\alpha +
\frac{n}{2},
$$

$$
C(\alpha) := \frac{1}{\Gamma(\alpha + 1)}, \ \ \ \alpha > -1,
$$
 $\Gamma(\cdot)$ is the Gamma-function, $J_{\lambda}(\cdot) $ is the Bessel function of order
$\lambda$ and $|z| = \sqrt{(z,z)}$.

Notice that, under our restriction on $\alpha$,

\begin{equation} \label{norming K}
\int_{\mathbb{R}^n} K_{\lambda}^R(t)\ dt = 1.
\end{equation}

\vspace{3mm}

Evidently,

\begin{equation} \label{dilatations}
K_{\lambda}^R(z) = R^n \ K_{\lambda}(Rz),
\end{equation}
where in turn

\begin{equation} \label{K repr}
 K_{\lambda}(z) \stackrel{def}{=} \ K_{\lambda}^1(z) =    C(\alpha) \frac{J_{\lambda}(|z|)}{|z|^{\lambda}},
\end{equation}
see, e.g., \cite[pp. 171-172]{Duoandikoetxea}.

Note that, from the identity (\ref{dilatations}), it follows

\begin{equation} \label{R identity}
||K_{\lambda}^R(z)||_q = R^{n - n/q} \  || K_{\lambda}||_q, \ \ \ q
\ge 1,
\end{equation}
as long as $ \ R > 0. \ $ \par

\vspace{4mm}

Assume here and in the sequel

 \begin{equation} \label{restr alpha}
-1 < \alpha \le \alpha_0, \hspace{5mm} \alpha_0  := \frac{n-1}{2},
\end{equation}
and introduce the value

\begin{equation} \label{qo}
q_0 := \frac{n}{(n+1)/2 + \alpha},
\end{equation}
so that
$$
q_0 \geq1 \ \ \Leftrightarrow  \ \ \frac{-(n+1)}{2} < \alpha \le
\alpha_0,
$$
where the inequality on the left hand side is true due to the
restriction $\alpha>-1$.

\vspace{4mm}

Moreover, using the well-known results about the behavior of the
Bessel functions (see, e.g., \cite[p. 172]{Duoandikoetxea}),
$$K_{\lambda}(z)\asymp |z|^{-(\frac{n+1}{2}+\alpha)}, \ \ \ {\rm as} \ |z|\to\infty ,$$
it is easy to estimate
\begin{equation}  \label{q norm}
||K_{\lambda}||_q \asymp (q - q_0)^{-1/q}, \ \ \ q > q_0,
\end{equation}
and $ \ ||K_{\lambda}||_q = \infty \ $  otherwise (see also, e.g.,
\cite[pp. 339-341]{Grafakos}).

Correspondingly,

\begin{equation} \label{qR norm}
||K^R_{\lambda}||_q \asymp  R^{n - n/q} \ (q - q_0)^{-1/q}, \ \ \ q
> q_0.
\end{equation}

\vspace{4mm}

Furthermore, we recall the Beckner-Brascamp-Lieb-Young inequality
for the convolution (see \cite{Beckner,Brascamp})
\begin{equation} \label{Beckner}
||f*g||_r \le \left(\frac{C_p C_q}{C_r}\right)^d \ ||f||_p \
||g||_q, \ \ \ \ 1 + 1/r=1/p + 1/q ,
\end{equation}
where $  p,q, r \geq 1$ and
$$
C_m = \left(\frac{m^{1/m}}{(m')^{1/m'}}\right)^{\frac{1}{2}}, \  \ \
m' = m/(m-1).
$$
\noindent Note that
$$
\frac{C_p C_q}{C_r} \le 1.
$$

Recently in \cite{Formica Ban Alg} has been given a generalization
of the convolution inequality in the context of the Grand Lebesgue
Spaces (see Section \ref{Section Boundedness GLS} for the
definition), built on a unimodular locally compact topological
group.

\vspace{2mm} The estimate \eqref{Beckner} is essentially
non-improbable. Indeed, the equality in (\ref{Beckner}) is attained
iff both functions $ f, g$
%and their convolution $f*g $
are proportional to Gaussian densities, namely there exists positive
constants $c_1,c_2,c_3,c_4$ such that
\begin{equation}\label{gaussian f g}
f(x) = c_1 \exp(-c_2||x||^2), \ \  \  \ g(x) = c_3
\exp(-c_4||x||^2), \ \ \  \ x \in \mathbb R^n,
\end{equation}
where $||\cdot||$ is the euclidean norm. Then the convolution $ f*g$
is also Gaussian and
\begin{equation} \label{Beckner inverse}
||f*g||_r = \left(\frac{C_p C_q}{C_r}\right)^d \ ||f||_p \ ||g||_q,
\  \ \ 1/p + 1/q = 1 + 1/r, \ \ \  p,q,r>1.
\end{equation}

\vspace{2mm} The following result about the boundedness of the
Bochner-Riesz operator, acting between different Lebesgue-Riesz
spaces, holds.

\begin{theorem}\label{boundedness Lebesgue}
% Let $f:\mathbb{R}^n\to\mathbb{R}$, \ $(n>1)$, be a measurable
%function,
Let $p>1$ and $f\in L_p(\mathbb{R}^n)$. Let $\alpha$ be a constant
such that
\begin{equation} \label{restr alpha}
 - 1  < \alpha \le \alpha_0 := \frac{n-1}{2}
\end{equation}
and
\begin{equation} \label{restrict n alpha}
q_0 = \frac{n}{(n+1)/2 + \alpha}.
\end{equation}

\noindent Let $q,r\geq 1$ such that $1 + 1/r = 1/q +1/p$ and assume
$q>q_0, \ r>r_0, \ p\leq p_0$, where
\begin{equation} \label{ro bound}
r_0 = r_0(\alpha,n; \, p) := \frac{p q_0}{p + q_0 - pq_0},
\end{equation}
\begin{equation}\label{p0}
p_0 =q_0'= (1 - 1/q_0)^{-1}
\end{equation}
Define, for $R = {\const} > 0$,
\begin{equation} \label{W introduce}
 W(\alpha,n, R; \ p,r) \stackrel{def}{=}  R^{n(1/p - 1/r)} \cdot [ \ q(p,r) - q_0(\alpha,n) \ ]^{1/p - 1- 1/r
 }.
\end{equation}
Then the Bochner-Riesz operator satisfies
\begin{equation} \label{Lr estim}
||B_R^{\alpha} [f]||_r \le  C(\alpha,R) \ \cdot W(\alpha, n, R; \
p,r) \ \cdot ||f||_p.
\end{equation}
\end{theorem}

\begin{proof}

If $f \in L_p$ for some value $ p>1$, then by \eqref{Beckner} it
follows
\begin{equation} \label{key estimate}
||B_R^{\alpha}[f]||_r \le ||K_{\lambda}^R||_q \ ||f||_p,
\end{equation}
where $1 + 1/r = 1/q +1/p \ $ and $ \ p, \ q, \ r \ge 1, \ q > q_0
$. On the other words,

$$
q = q(p,r) \stackrel{def}{=} \frac{pr}{pr + p - r }, \ \ \ q > q_0,
\ \ \ r> p.
$$

\noindent Moreover it is easy to verify that $r_0$, defined in
\eqref{ro bound}, is such that $r_0> p$.

\noindent Under our restrictions $W(\alpha,n, R; \ p,r)$ is finite
and positive.

So we conclude that estimate \eqref{Lr estim} holds.

\end{proof}

\vspace{5mm}

\section{Main result. Boundedness of Bochner-Riesz operators in Grand Lebesgue Spaces (GLS)
}\label{Section Boundedness GLS}

\vspace{5mm}

We recall here, for reader convenience, some known definitions and
facts  from the theory of Grand Lebesgue Spaces (GLS).

\begin{definition}
 {\rm Let $ \ \psi = \psi(p), \ p \in (a,b)$, \ $ a,b = {\const}, \ 1 \le  a < b \le \infty$, be a positive measurable numerical valued
 function, such that
 \begin{equation} \label{Positive psi}
 \inf_{p \in (a,b)} \psi(p) > 0.
\end{equation}
The (Banach) Grand Lebesgue Space $G\psi =G\psi(a,b)$ consists of
all the real (or complex) numerical valued measurable functions $f:
\mathbb{R}^n \to \mathbb R$ having finite norm, defined by
\begin{equation} \label{norm psi}
    || f ||_{G\psi} :=\sup_{p \in (a,b)} \left[ \frac{||f||_p}{\psi(p)} \right].
 \end{equation}
We agree to write $ \ G\psi \ $ in the case when $a = 1$ and $b =
\infty$.

\noindent  The function $\psi$ is named {\it generating function}
for the space $G\psi$ and we denote by $\{\psi(\cdot)\}$ the set of
all such functions. }
\end{definition}

\noindent For instance
\begin{equation}\label{psi power}
\psi(p) := p^{1/m},  \  \ \ p \in [1,\infty), \ \ \  m > 0,
\end{equation}
or
\begin{equation}\label{psi grand IwSb}
\psi(p) := (p-a)^{-\alpha} \ (b-p)^{-\beta}, \ \ \ p \in (a,b),  \ \
\  1\leq a<b<\infty, \ \ \ \alpha,\beta \ge 0,
\end{equation}
are generating functions.

\vspace{2mm}

If
 %\begin{equation*}
% \psi(p) =\left\{
%  \begin{array}{ll}
%    1  & , \ \ p=r \\
%    +\infty &  , \ \ p\neq r
%  \end{array}
%\right. , \ \ \ \ r\in [1,\infty)
%\end{equation*}
$$
 \psi(p) = 1, \ \ p = r;  \  \ \ \ \psi(p) = +\infty,  \ \  p \ne
 r, \ \ \ \ r\in[1,\infty),
$$
where  $C/\infty := 0, \ C \in \mathbb R$ (extremal case), then the
corresponding $ G\psi$ space coincides with the classical
Lebesgue-Riesz space $L_r = L_r(\mathbb{R}^d)$.

\vspace{4mm}

 The Grand Lebesgue Spaces and several generalizations of them have been widely investigated, mainly in the case of GLS on sets of finite measure, (see, e.g.,
  \cite{Ermakov etc. 1986,Iwaniec1,Iwaniec2,Fiorenza2000,Lifl,Ostrovsky1,caponeformicagiovanonlanal2013,formicagiovamjom2015}, etc).
They play an important role in the theory of Partial Differential
Equations (PDEs) (see, e.g.,
\cite{Greco-Iwaniec-Sbordone-1997,Fiorenza-Formica-Gogatishvili-DEA2018,fioformicarakodie2017,Ahmed
Fiorenza Formica at all}, etc.), in interpolation theory (see, e.g.,
\cite{fioforgogakoparakoNA,fiokarazanalanwen2004}), in the theory of
Probability (\cite{Ostrovsky3,Ostrovsky5,ForKozOstr_Lithuanian}), in
Statistics \cite[chapter 5]{Ostrovsky1}, in theory of random fields
\cite{KozOs}, \cite{Ostrovsky4}, in Functional Analysis
\cite{Ostrovsky1}, \cite{Ostrovsky2}, \cite{Ostrovsky4} and so one.
\par

These spaces are rearrangement invariant (r.i.) Banach functional
(complete) spaces; their fundamental functions have been considered
in \cite{Ostrovsky4}. They do not coincide, in the general case,
with the classical Banach rearrangement functional  spaces: Orlicz,
Lorentz, Marcinkiewicz  etc., see \cite{Lifl,Ostrovsky2}. The
belonging of a function $ f: \mathbb{R}^n \to \mathbb{R}$ to some $
G\psi$ space is closely related with its tail function behavior as $
\ t \to 0+ \ $ as well as when $ \ t \to \infty,  \ $ see
\cite{KozOs,KozOsSir2017}.

\vspace{2mm}

The Grand Lebesgue Spaces can be considered not only on the
Euclidean space $ \mathbb{R}^n$ equipped with the Lebesgue measure,
but also on an arbitrary measurable space with sigma-finite
non-trivial measure.

\vspace{2mm}

In the following Theorem we investigate the boundedness of the
Bochner-Riesz operator acting from some Grand Lebesgue Space $ G\psi
=G\psi(a,b)$ into another one $ G\nu$. We will consider the same
restrictions and quantities defined in Theorem \ref{boundedness
Lebesgue}.

 \vspace{2mm}

\begin{theorem}
Let $1\leq a<b\leq \infty$ and $f\in  G\psi(a,b)$. Let $\alpha$ be a
constant such that
\begin{equation*}
 - 1  < \alpha \le \alpha_0 := \frac{n-1}{2}, \ \ \ n>1.
\end{equation*}

\noindent Let $p>1$, $q, r\geq 1$, \  $p_0, q_0, r_0$ and $
W(\alpha,n, R; \ p,r)$ defined as in Theorem \ref{boundedness
Lebesgue}.

\noindent Denote
$$
s = s(\alpha,n;r) := \min \left\{ b, \ \frac{r q_0}{r q_0 + q_0 - r}
\right\},
$$
and
$$d := r_0(\alpha,n; \, b)=\frac{bq_0 }{q_0 +b - b q_0}.$$

\noindent For $r\in(d,\infty)$ let $\nu(r)$ be the following
generating function
\begin{equation} \label{nu function}
\nu(r) = \nu[\psi](r) = \nu[\psi](\alpha,n,R; r) := \inf_{p \in
(a,s)} [ \ W(\alpha,n, R; \, p,r) \ \psi(p) \ ].
\end{equation}
%which is defined (and is finite)
Then
\begin{equation} \label{main res}
|| B_R^{\alpha}[f]||_{G\nu} \le C(\alpha,n,R) \ ||f||_{G\psi}.
\end{equation}
\end{theorem}

\noindent \begin{proof}

The proof is simple and alike as the one in \cite{Ostrovsky5}. First
we observe that $\nu(r)$ is finite. One can assume, without loss of
generality, $ ||f||_{G\psi} = 1$, then $ ||f||_p \le \psi(p), \ p
\in (a,b)$.
 Applying the inequality (\ref{Lr estim}) we have
\begin{equation} \label{Gpsi estim}
||B_R^{\alpha}[f]||_r \le  C(\alpha, n,R) \ \cdot W(\alpha,n, R; \ p,r) \ \cdot \psi(p).
\end{equation}

Taking the minimum over $p$ subject to  our limitations, we get
\begin{equation} \label{Gpsi min estim1}
\begin{split}
||B_R^{\alpha} [f]||_r  & \le  C(\alpha,n,R) \, \cdot \inf_{p \in
(a,s)} [ \ W(\alpha,R; \, p,r) \ \cdot \psi(p) \ ]\\
&= C(\alpha,n,R) \ \nu(r) = C(\alpha,n,R) \ \nu(r) ||f||_{G\psi},
\end{split}
\end{equation}
which is quite equivalent to our claim in (\ref{main res}).

\end{proof}

\vspace{5mm}

\section{Lower bound for the coefficient in the Lebesgue-Riesz norm estimate
for the Bochner-Riesz operator.}\label{Section lower bound BR}

%{Lower bound for the norm of the Bochner-Riesz operator}

\vspace{5mm}

Let $p,r>1$, $n>1$ and let us introduce the following variable

$$
Q_n(p,r) \stackrel{def}{=} \sup_{\alpha > 0} \sup_{R > 0}
W(\alpha,n,R; \ p,r),
$$
where $W(\alpha,n,R; \ p,r)$ is defined in Section \ref{estimates
Lebesgue}. Our target in this Section is a {\it lower} bound for the
above variable.

\vspace{2mm}

\begin{theorem}\label{Th lower bound}
Let $p,r> 1$, $n>1$ and
$$
\Theta(n,p) := (2 \ \pi)^{ n(1 - p)/2p  } \ p^{-n/2p}.
$$
Then
\begin{equation} \label{lower bound}
Q_n(p,r) \ge \Theta \left(n, \frac{pr}{pr + p - r} \ \right), \  \ \
r> p.
\end{equation}

\end{theorem}

\vspace{2mm}

\begin{remark}
{\rm The {\it possible} case  when $ \ Q_n(p,r) = + \infty \ $ can
not be excluded.}
\end{remark}

\vspace{2mm}

\noindent \begin{proof}

We will apply equality \eqref{Beckner inverse}, in which we choose
the ordinary Gaussian density

$$
f_0(x) := (2 \pi)^{-n/2} \exp(-||x||^2/2), \ \  x \in \mathbb{R}^n,
$$
and take $ \ \alpha = R^2/2. \ $  Obviously

$$
Q_n(p,r) \ge  \lim_{R \to \infty}  W(R^2/2,n,R; \, p,r).
$$
We have
$$
B_R^{R^2}[f](t) = \int_{\mathbb{R}^n} e^{-i(t,y)} \left( \ 1 -
\frac{||y||^2}{R^2} \ \right)^{R^2/2}  I(||y|| < R) \ \tilde{f}(y) \
dy ,
$$
where $ \ I(A)  \ $ denotes the indicator function of the
(measurable) set  $ \ A, \ A \subset \mathbb{R}^n$.

\noindent Therefore, as $ R \to \infty \ $,

\begin{equation} \label{f conv f0}
B_R^{R^2}[f](t)\to \int_{\mathbb{R}^n} e^{-i(t,y)} \ \exp \{ \ -
||y||^2/2 \ \} \ dy  \ \tilde{f}(y) \ dy =[\sqrt{2 \pi} ]^n \
[f*f_0](t),
\end{equation}
by virtue of dominated convergence theorem.

\vspace{3mm}

\noindent If we take $ \ f = f_0, \ $ then in \eqref{f conv f0} the
convolution of two Gaussian densities appears. \par

\noindent It remains to apply the relation \eqref{Beckner inverse};
we omit some simple calculations.\par

 \end{proof}

\vspace{5mm}

\section{Convergence of Bochner-Riesz operators.}\label{section
convergence}

\vspace{5mm}

 We investigate here the convergence, as $ R \to \infty $, of the family of Bochner-Riesz approximations $ \ B_R^{\alpha}[f] \ $ to
 the source function $f$ in the Lebesgue-Riesz norm $ \ L_p(\mathbb{R}^n), \ p \in (1,\infty)$, in addition to the  similar results in \cite{Grafakos,Fefferman2,Karasev},
 etc.
\par
 For any function $ f\in L_p(\mathbb{R}^n) $, its modulus of $ \ L_p \ $ continuity is defined alike as in
 approximation theory \cite[chapter V]{Achieser} :

$$
\omega_p[f](\delta) \stackrel{def}{=} \sup_{h \, : \ |h| \le \delta}
||T_h[f] - f||_p, \ \ \  \delta \ge 0,
$$
where $ \ T_h[f] \ $ denotes the shift operator
$$
T_h[f](t) =f(t - h), \ \ \  t,h \in \mathbb{R}^n.
$$

 Obviously, $ \ \omega_p[f](\delta) \le 2 ||f||_p \ $ and

\begin{equation} \label{omega delta}
\lim_{\delta \to 0+} \omega_p[f](\delta) = 0, \ \ \ f \in L_p.
\end{equation}

\vspace{2mm}

\begin{theorem}\label{Th convergence p finite}

 Let $p \in (1,\infty), \ \alpha  \leq \alpha_0 = (n-1)/2 $  \, and $ \ f \in L_p(\mathbb{R}^n)$. Then

\begin{equation} \label{convergence}
\lim_{R \to \infty} ||B_R^{\alpha}[f] - f||_p  = 0.
\end{equation}

\end{theorem}

\noindent\begin{proof}

The difference $ \ \Delta_R[f](t) = B_R^{\alpha} [f]-f \ $ has the
form

$$
\Delta_R[f](t) = \int_{\mathbb{R}^n} \left\{ \ f \left(t -
\frac{v}{R} \right)  - f(t) \ \right\} \ K_{\lambda}(|v|) \ dv.
$$
 We apply  now the triangle inequality for the $ \ L_p \ $ norm in the integral form

\begin{eqnarray*}
\begin{split}
||\Delta_R[f]||_p & \le  \int_{\mathbb{R}^n} ||T_{|v|/R}[f] - f||_p
\ |K_{\lambda}(|v|)| \ dv \\
& \leq \int_{\mathbb{R}^n} \omega_p[f](|v|/R) \ |K_{\lambda}(|v|)| \
dv.
\end{split}
\end{eqnarray*}

 Note that, under the above conditions,

$$
\int_{\mathbb{R}^n} |K_{\lambda}(|v|)| \ dv < \infty,
$$
so that \eqref{convergence} follows again from the dominated
convergence theorem.
\end{proof}

\begin{remark}

{\rm As a slight consequence, under the above conditions, if $ f \in
L_p(\mathbb{R}^n) $, then

$$
||B_R^{\alpha}[f]||_p \le ||f||_p + \int_{\mathbb{R}^n}
\omega_p[f](|v|/R) \ |K_{\lambda}(|v|)| \ dv
$$
and, consequently,
$$
\forall f \in L_p(\mathbb{R}^n) \ \ \Rightarrow \ \ \sup_{R \ge 1}
||B_R^{\alpha}[f]||_p \le 3 ||f||_p\,.
$$
}
\end{remark}

\vspace{4mm}

 The case $ \ p = \infty \ $ requires a separate consideration. Introduce the Banach space $ \ C_0(\mathbb{R}^n) \ $ as a collection of all bounded and
 {\it uniformly continuous} functions $ \ f: \mathbb{R}^n \to \mathbb{R}\ $,  equipped with the ordinary
 norm,
$$
||f||_{\infty} \stackrel{def}{=} \sup_{t \in \mathbb{R}^n}|f(t)|.
$$
As above
$$
\omega_{\infty}[f](\delta) \stackrel{def}{=} \sup_{h \,: \ |h| \le
\delta} ||T_h[f] - f||_{\infty}, \ \ \ \delta \ge 0.
$$
Evidently,  $ \ \omega_{\infty}[f](\delta) \le 2 ||f||_{\infty} \ $
and
\begin{equation} \label{omega delta infty}
\lim_{\delta \to 0+} \omega_{\infty}[f](\delta) = 0,  \ \ \ f \in
L_p.
\end{equation}

\vspace{4mm}

The assertion of Theorem \ref{Th convergence p finite} under the
same conditions remains true in the case $ \ p = \infty. \ $

\begin{theorem}\label{Th convergence p infty}

Under the same conditions of Theorem \ref{Th convergence p finite},
for any function $ \ f: \mathbb{R}^n \to \mathbb{R} \ $ from the
space $ \ C_0(\mathbb{R}^n) \ $,  its  Bochner-Riesz approximation $
\ B_R^{\alpha} \ $ converges uniformly to the source function $ f$,
that is
\begin{equation} \label{convergence infty}
\lim_{R \to \infty} ||B_R^{\alpha}[f] - f||_{\infty}  = 0.
\end{equation}

\end{theorem}

\vspace{5mm}

\noindent \begin{proof} The proof is the same as in Theorem \ref{Th
convergence p finite} and may be omitted.
\end{proof}

\begin{remark}
As before, if $f\in C_0(\mathbb{R}^n) $, then
$$
||B_R^{\alpha}[f]||_{\infty} \le ||f||_{\infty} +
\int_{\mathbb{R}^n} \omega_{\infty}[f](|v|/R) \ |K_{\lambda}(|v|)| \
dv
$$
and moreover
$$
\forall f \in C_0(\mathbb{R}^n) \ \ \Rightarrow \ \ \sup_{R \ge 1}
||B_R^{\alpha}[f]||_{\infty} \le 3 ||f||_{\infty}.
$$
\end{remark}

\vspace{5mm}

\section{Concluding remarks.}\label{Section conclusion}

\vspace{5mm}

 \hspace{6mm}  {\bf A.} In our opinion, the method described in this paper may be essentially generalized on more operators of convolutions type, linear or not.
See some preliminary results \cite{Lifl}.

\vspace{2mm}

 \ {\bf B.} It is interesting to generalize the estimates obtained in the previous Sections to the so-called {\it maximal} operators
 associated with the Bochner-Riesz one considered here, in the spirit of the works \cite{fioguptajainstudiamath2008,Jain Kumari},
and so one:

$$
||\sup_{R \ge 1} B_R^{\alpha}[f]||_p  \le M(\alpha,n;p) \cdot
||f||_p, \ \  \ p \ge p_0,
$$
or

$$
||\sup_{R \ge 1} B_R^{\alpha}[f]||_r  \le M(\alpha,n; p,r) \cdot
||f||_p, \ \ p \ge p_0, \ \ \ r = r(p),
$$
in order to obtain the GLS estimate  for the Bochner-Riesz {\it
maximal} operator of the form

$$
||\sup_{R \ge 1} B_R^{\alpha}[f]||_{G\zeta}   \le L(\alpha,n;
\zeta,\psi) \cdot ||f||_{G\psi}
$$
for some generating functions  $ \ \psi, \ \zeta. \ $ \par

\vspace{5mm}

\vspace{0.5cm} \emph{Acknowledgement.} {\footnotesize The first
author has been partially supported by the Gruppo Nazionale per
l'Analisi Matematica, la Probabilit\`a e le loro Applicazioni
(GNAMPA) of the Istituto Nazionale di Alta Matematica (INdAM) and by
Universit\`a degli Studi di Napoli Parthenope through the project
\lq\lq sostegno alla Ricerca individuale\rq\rq .\par


\vspace{6mm}

\begin{thebibliography}{55}

\bibitem{Achieser}
 N.I. Achiezer. {\it Theory of approximation,} F. Ungar, (1956), (Translated from Russian).

\bibitem{Ahmed Fiorenza Formica at all}
I. Ahmed, A. Fiorenza, M.R. Formica, A. Gogatishvili and J.M.
Rakotoson. \emph{Some new results related to Lorentz G-Gamma spaces
and interpolation.} J. Math. Anal. Appl., 483, {\bf 2}, 123623
(2020).

%\bibitem{Beals}
%M. Beals, $L^ p$ \emph{boundedness of Fourier integrals.}  Mem.
%Amer. Math. Soc. \textbf{38}, no. 264 (1982),

\bibitem{Beckner}
W. Beckner, \emph{Inequalities in Fourier analysis}, Ann. of Math.
\textbf{102} (1975), 159--182.


%\bibitem{Borjeson}
%L. B\"{o}rjeson \emph{Estimates for the Bochner-Riesz operator with
%negative index.} Indiana Univ. Math. J. \textbf{35}, no. 2, (1986),
%225--233.


\bibitem{Brascamp}
H.J. Brascamp and  E.H. Lieb, \emph{Best constants in Young's
inequality, its converse, and its generalization to more than three
functions}, Advances in Math. \textbf{20} (2) (1976), 151--173.

\bibitem{Buldygin Ill}
V.V. Buldygin and E.V. Ilarionov, \emph{A problem in the statistics
of random fields}. (Russian)  Probabilistic infinite-dimensional
analysis, pp. 6--14, 123, Akad. Nauk Ukrain. SSR, Inst. Mat., Kiev,
1981.


%\bibitem{Buld Mush OsPuch1992}
%{\bf Buldygin V.V., Mushtary D.I., Ostrovsky E.I, Pushalsky M.I.}
%\emph{New Trends in Probability Theory and Statistics.} Mokslas,
%1992, Amsterdam, New York, Tokyo.

\bibitem{caponeformicagiovanonlanal2013}
C. Capone, M.R. Formica and R. Giova, \emph{Grand Lebesgue spaces
with respect to measurable functions}, Nonlinear Anal. \textbf{85}
(2013), 125--131.

\bibitem{Cordoba1979}
A. C\'{o}rdoba, \emph{A note on Bochner-Riesz operators}, Duke Math.
J. \textbf{46} (1979), no. 3, 505--511.

%\bibitem{Coriasco1}
%S. Coriasco and M. Ruzhansky, \emph{On the boundedness of Fourier
%integral operators on} $L^p(R^n)$, C. R. Math. Acad. Sci. Paris
%\textbf{348} no. 15-16, (2010), 847--851.
%
%\bibitem{Coriasco2}
% S. Coriasco  and  M. Ruzhansky, \emph{Global} $L^p$
%\emph{continuity of Fourier  integral  operators}, Trans. Amer.
%Math. Soc. \textbf{366}, no. 5, (2014),  2575--2596.


\bibitem{Duoandikoetxea}
J. Duoandikoetxea, \emph{Fourier analysis}, Graduate Studies in
Math., vol. 29. American Mathematical Society, Providence, RI, 2001.
Translated and revised from the 1995 Spanish original by David
Cruz-Uribe.


%\bibitem{Edmunds}
%D.E. Edmunds and B. Opic, \emph{Equivalent quasi-norms on Lorentz
%spaces}, Proc. Amer. Math. Soc. \textbf{131}, no. 3, (2003),
%745--754.

\bibitem{Ermakov etc. 1986}
S. V. Ermakov, and E. I. Ostrovsky, \emph{Continuity Conditions,
Exponential Estimates, and the Central Limit Theorem for Random
Fields}, Moscow, VINITY,  1986 (in Russian).


\bibitem{Fefferman1}
 C. Fefferman, \emph{Inequalities for strongly singular
convolution operators}, Acta Math. \textbf{124} (1970), 9--36.

\bibitem{Fefferman2}
C.Fefferman, \emph{A note on spherical summation multipliers},
Israel J. Math., {\bf 15,} (1973), 44-- 52.

\bibitem{Fiorenza2000}
A. Fiorenza, \emph{Duality and reflexivity in grand Lebesgue
spaces}, Collect. Math. \textbf{51} (2) (2000), 131--148.


\bibitem{Fiorenza-Formica-Gogatishvili-DEA2018}
A. Fiorenza, M.R. Formica and A. Gogatishvili, \emph{On grand and
small Lebesgue and Sobolev spaces and some applications to PDE's}.
Differ. Equ. Appl., \textbf{10} (1), (2018), 21--46.

\bibitem{fioforgogakoparakoNA}
A. Fiorenza, M.R. Formica, A. Gogatishvili, T. Kopaliani and J.M.
Rakotoson, \emph{Characterization of interpolation between grand,
small or classical Lebesgue spaces}, Nonlinear Anal. \textbf{177}
(2018), 422--453.


\bibitem{fioformicarakodie2017}
A. Fiorenza, M.R. Formica and J.M. Rakotoson, \emph{Pointwise
estimates for {$G\Gamma$}-functions and applications}, Differential
Integral Equations \textbf{30} (11-12) (2017), 809--824.


\bibitem{fioguptajainstudiamath2008}
A. Fiorenza, B. Gupta and P. Jain, \emph{The maximal theorem for
weighted grand Lebesgue spaces}, Studia Math. \textbf{188}, no.2,
(2008), 123--133.

\bibitem{fiokarazanalanwen2004}
A. Fiorenza and G.E. Karadzhov, \emph{Grand and small Lebesgue
spaces and their analogs}, Z. Anal. Anwend. \textbf{23} (4) (2004),
657--681.

%\bibitem{Fiorenza2018a}
%{\bf A. Fiorenza, M.R. Formica, T. Roskovec, F. Soudsky},
%Gagliardo-Nirenberg Inequality for rearrangement-invariant Banach
%function spaces. \emph{Atti Accad. Naz. Lincei Rend. Lincei Mat.
%Appl.} 30 (2019),  no. 4, 847--864.

%
%\bibitem{Fiorenza2018b}
%{\bf A.Fiorenza, M.R.Formica, T.Roskovec and F.Soudsky.} {\it Detailed proof of classical  Gagliardo - Nirenberg inequality interpolation
%inequality with historical remarks.}
%arXiv:1812.04281v1  [math.FA]  11 Dec 2018

\bibitem{formicagiovamjom2015}
M.R. Formica and R. Giova, \emph{Boyd indices in generalized grand
Lebesgue spaces and applications}, Mediterr J. Math., \textbf{12}
(3) (2015), 987--995.


\bibitem{ForKozOstr_Lithuanian}
M.R. Formica, Y.V. Kozachenko, E. Ostrovsky, L. Sirota,\emph{
Exponential tail estimates in the law of ordinary logarithm (LOL)
for triangular arrays of random variables}. {Lith. Math. J.} (2020),
DOI 10.1007/s10986-020-09481-x.

%\bibitem{Formica As or not}
%M.R. Formica, E. Ostrovsky, L. Sirota, \emph{Asymptotic and
%non-asymptotic estimates for multivariate Laplace integrals}. {Stat.
%Optim. Inf. Comput.} 7(4) (2019), 759--778.


\bibitem{Formica Ban Alg}
M.R. Formica, E. Ostrovsky, L. Sirota, {\it Grand Lebesgue Spaces
are really Banach algebras relative to the convolution on unimodular
locally compact groups equipped with Haar measure}, to appear on
Math. Nachr.

\bibitem{Grafakos}
L. Grafakos, \emph{Modern Fourier analysis}, Third edition. Graduate
Texts in Mathematics, 250. Springer, New York, 2014.

\bibitem{Greco-Iwaniec-Sbordone-1997}
L. Greco, T. Iwaniec and C. Sbordone, \emph{Inverting the
$p$-harmonic operator}, Manuscripta Math. \textbf{92}, no. 2,
(1997), 249--258.


\bibitem{Iwaniec1}
T. Iwaniec and C. Sbordone, \emph{On the integrability of the
Jacobian under minimal hypotheses}, Arch. Rational Mech. Anal.
\textbf{119}, no. 2, (1992),129--143.

\bibitem{Iwaniec2}
T. Iwaniec, P. Koskela and J. Onninen,  \emph{Mappings of finite
distortion: monotonicity and continuity}, Invent. Math.
\textbf{144}, no. 3, (2001), 507--531.

\bibitem{Jain Kumari}
P. Jain  and S. Kumari, \emph{On grand Lorentz spaces and the
maximal operator}. Georgian Math. J. \textbf{19}, no. 2, (2012),
235--246.


\bibitem{Karasev}
D.N. Karasev and V.A. Nogin, $L_p-L_q$ \emph{estimates for the
Bochner-Riesz operator of complex order}, Z. Anal. Anwendungen
\textbf{21}, no. 4, (2002), 915--929.


\bibitem{KozOs}
Yu.V. Kozachenko and E.I. Ostrovsky, \emph{Banach Spaces of random
variables of sub-Gaussian type}, Teor. Veroyatn. Mat. Stat., Kiev,
\textbf{32} (134), (1985), 42--53, (in Russian). English transl.:
Theory Probab. Math. Stat., \textbf{32}, (1986), 45--56.

\bibitem{Kozachenko at all 2018}
Yu.V. Kozachenko, Yu.Yu. Mlavets, and N.V. Yurchenko, \emph{Weak
convergence of stochastic processes from spaces} $F_\psi(\Omega)$,
Statistics, Optimization \& Information Computing, \textbf{6} (2),
(2018), 266--277.

\bibitem{KozOsSir2017}
Yu.V. Kozachenko, E. Ostrovsky and L. Sirota,
\emph{Relations between exponential tails, moments and moment generating functions for random variables and vectors}. \\
arXiv:1701.01901v1 [math.FA] 8 Jan 2017

%\bibitem{KozOsSir oct2017}
%Yu.V. Kozachenko, E. Ostrovsky and L. Sirota,
%\emph{Equivalence between tails, Grand Lebesgue Spaces and Orlicz norms for random variables without Cramer's condition}. \\
%arXiv:1710.05260v1  [math.PR]  15 Oct 2017

\bibitem{Lifl}
E. Liflyand, E. Ostrovsky and L.Sirota, \emph{Structural properties
of bilateral grand Lebesque spaces}, Turkish J. Math. \textbf{34},
no. 2, (2010), 207--219.

\bibitem{LuYanBook2013}
S. Lu and D. Yan, \emph{Bochner-Riesz means on Euclidean spaces},
World Scientific Publishing Co. Pte. Ltd., Hackensack, NJ, 2013, 376
pp.

%\bibitem{Ostrovsky0}
%E.I. Ostrovsky, \emph{Exponential Estimations for Random Fields},
%OINPE, Moscow, Obninsk, 1999.

\bibitem{Ostrovsky1}
E. Ostrovsky, \emph{Exponential Orlicz's spaces: new norms and
applications}.\\
 arXiv/FA/0406534, v.1, (25.06.2004.)

\bibitem{Ostrovsky2}
E. Ostrovsky and L. Sirota, \emph{Some new rearrangement invariant spaces:
theory and applications}. \\
 arXiv:math.FA/0605732 v1, 29, (May 2006).

\bibitem{Ostrovsky3}
E. Ostrovsky and L. Sirota, \emph{Fourier Transforms in Exponential
Rearrangement Invariant Spaces}. \\
arXiv:math.FA/040639, v1, (20.6.2004.)

\bibitem{Ostrovsky4}
E. Ostrovsky and L. Sirota, \emph{Moment Banach spaces: Theory and
applications}, HIAT Journal of Science and Engineering C, 4(1--2),
(2007), 233--262.

\bibitem{Ostrovsky5}
E. Ostrovsky and L. Sirota, \emph{Boundedness of operators in
bilateral Grand Lebesgue Spaces, with exact and weakly exact
constant calculation}. arXiv:1104.2963v1 [math.FA] 15 Apr 2011.


\bibitem{Ostrovsky6}
E. Ostrovsky , S.Yu Tsykunova, \emph{Asymptotic properties of the
distribution of the maximum of a Gaussian nonstationary process that
arises in covariance statistics}, (Russian) Teor. Veroyatnost. i
Primenen. 39 (1994),  no. 3, 641--649;  translation in  Theory
Probab. Appl. \textbf{39} (1994), no. 3, 527--534 (1995).


\bibitem{Stein1}
E.M. Stein, \emph{Singular integrals and differentiability
properties of functions}, Princeton Mathematical Series, no. 30
Princeton University Press, Princeton, N.J. 1970.

\bibitem{Stein2}
E.M. Stein, \emph{Harmonic analysis: real-variable methods,
orthogonality, and oscillatory integrals}. With the assistance of
Timothy S. Murphy. Princeton Mathematical Series, 43. Monographs in
Harmonic Analysis, III. Princeton University Press, Princeton, NJ,
 1993.

%\bibitem{Talenti}
%G. Talenti, \emph{Inequalities in rearrangement invariant function
%spaces}, Nonlinear analysis, function spaces and applications, Vol.
%5 (Prague, 1994), 177-230, Prometheus, Prague, 1994.







\end{thebibliography}
\end{document}